\documentclass{siamltex}
\usepackage{bbm}
\usepackage{flafter}
\usepackage{colortbl}
\usepackage{fleqn}
\usepackage{amsmath,mathrsfs}
\usepackage{amsbsy}
\usepackage{latexsym}
\usepackage{psfrag}
\usepackage{graphicx,subfigure}
\usepackage{amssymb}
\usepackage{float}
\usepackage{amsfonts}
\usepackage{hyperref}%\href{url}{text}
\usepackage{url}
\usepackage{graphicx}
\usepackage{epstopdf}
\usepackage{epsfig}

\newtheorem{assumption}[theorem]{Assumption}

\newtheorem{remark}[theorem]{Remark}

\title{Event-triggered robust control of linear systems: Sliding mode cone method\thanks{This work was supported in part by the National Natural Science Foundation of China under Grant 62173308 and Grant 62303422; in part by the Natural Science Foundation of Zhejiang Province of China under Grant LRG25F030002 and Grant LQ23F030006; in part by Zhejiang Province Leading Geese Plan under Grant 2025C01056; and in part by Jinhua Science and Technology Project under Grant 2022-1-042.}}

\author{Bangxin Jiang\thanks{School of Computer Science and Technology, Zhejiang Normal University, Jinhua 321004, China ({\tt jiangbangxin@zjnu.edu.cn}).}
\and Ying Liu\thanks{School of Mathematical Sciences, Zhejiang Normal University, Jinhua 321004, China ({\tt liuying59@zjnu.edu.cn}).}
\and Yang Liu\thanks{Corresponding author. School of Mathematical Sciences, Zhejiang Normal University, Jinhua, 321004, China, and also the Hangzhou School of Automation, Zhejiang Normal University, Hangzhou, 311231, China ({\tt liuyang@zjnu.edu.cn}).}
\and Jianquan Lu\thanks{Department of Systems Science, School of Mathematics, Southeast University, Nanjing 210096, China ({\tt jqluma@gmail.com}).}
\and Weihua Gui\thanks{School of Automation, Central South University, Changsha, 410083, China ({\tt gwh@csu.edu.cn})}.
}
\date{}
\begin{document}

\maketitle
\begin{abstract}
In this paper, we investigate the global robust stabilization of linear time-invariant systems by using event-triggered sliding mode control (SMC). Different from the practical sliding mode band, which is commonly used in previous studies on event-triggered SMC, a new concept of ideal sliding mode cone is proposed in this paper. Specifically, we design a hybrid event-triggering mechanism that takes into account both the size and direction shift of the error state. The proposed event-triggered SMC law is shown to enforce and sustain the system state in the ideal sliding mode cone. Moreover, the state of the closed-loop system can asymptotically converge to the equilibrium point, rather than merely to a neighborhood of it, which is usually difficult to handle by using practical sliding mode band. Technically speaking, to achieve strong convergence, the triggering frequency should naturally be as high as possible due to the existence of the external disturbances, but this will also increase the communication load. Hence, to balance the asymptotic convergence and frequent triggering near the equilibrium point that is the price paid for achieving asymptotic stability, we extend the obtained results to the case of practical sliding mode cone. In addition, it is verified that the ETM is global, namely, the inter-event times are uniformly lower bounded from zero globally. Further, a practical application for the quadrotor unmanned aerial vehicles is presented. Finally, three illustrative examples are given to demonstrate the effectiveness of the obtained results.
\end{abstract}



\pagestyle{myheadings} \thispagestyle{plain} \markboth{Bangxin Jiang, Ying Liu, Yang Liu, Jianquan Lu and Weihua Gui}{Event-triggered robust control of linear systems:  Sliding mode cone method}

\section{Introduction}
Sliding mode control (SMC) is a control strategy known for its robustness, ensuring stable system performance in the presence of external perturbations and uncertainties \cite{shtessel2014sliding}. The core idea behind SMC is to design an appropriate sliding surface and a corresponding controller, so that the state of system reaches this predesigned surface in finite time and then slides along it asymptotically towards the equilibrium point \cite{edwards1998sliding,1101446,doi:10.1080/00207179208934240}. Once the system state reaches this sliding surface, it becomes highly insensitive to parameter variations and disturbances, ensuring reliable and consistent control performance. Due to these characteristics, research on SMC is actively conducted across a wide range of fields, concentrating on various aspects of this control technique \cite{doi:10.1137/22M1535309,doi:10.1137/1.9781611975840,li2021sliding,niu2005robust,polyakov2023homogeneous,yang2012sliding,10323267}. Up to now, SMC has been applied to many engineering systems, such as wind energy conversion systems \cite{en15051625,MOUSAVI2022112734}, rigid spacecraft and robot systems \cite{9774990,liu2018event}, unmanned marine vehicle systems \cite{8705012,yan2018sliding}, etc.

Time-triggered control strategies typically contain periodic execution intervals \cite{doi:10.1137/23M1571502}, defined upper and lower bounds, as well as some global conditions such as average dwell time \cite{doi:10.1137/20M1317037} and persistent dwell time (see \cite{831330} and \cite{7063896}). However, these time-triggered control strategies do not make use of the information of the system state, which leads to needless heavy workloads when computational and communication resources may be more usefully assigned to some other tasks. Research interest in event-triggered control (ETC) has arisen because of these limitations of time-triggered control \cite{liu2020event,wu2022dynamic}. Technically
speaking, the control signals only need to be updated at some
discrete instants, which are generated by a proposed state dependent
condition and such a condition is often called the event-triggering mechanism (ETM). Compared with time-triggered control, the event-triggered control strategy can effectively mitigate the unwarranted expenditure of control and communication resources. Thus, the research on ETC is of great practical significance (see \cite{heemels2012periodic,10428061,jiang2024periodic}).

Event-triggered sliding mode control combines both the robustness of SMC and the efficiency of ETC, and has therefore been extensively investigated \cite{chen2025design,cucuzzella2020event,10201802,kumari2020event,liu2018event,ma2021event,9204801,yan2023continuous}. In this strategy, transmissions are often triggered by changes of the system state, typically when the error state exceeds a threshold, rather than happening at predetermined instants \cite{jiang2021event}. And the practical sliding mode band is achieved by using event-triggered SMC, that is, any desired band size around the sliding manifold can be obtained. Generally speaking, in event-triggered SMC, the discretization of the control signal leads to the case that the state of the system cannot be continuous motion on the surface, but oscillates within a practical sliding mode band (see \cite{behera2021survey,bandyopadhyay2018event,7048489,cucuzzella2020event,kumari2020event}). The size of this band is primarily determined by the parameters inherent in the event-triggering mechanism \cite{7048489}. Furthermore, methods on the basis of the practical sliding mode band have yielded a number of interesting results and applications \cite{7448882,BEHERA201861,liu2018event,ye2022event}. However, previous studies on event-triggered SMC utilized the idea of practical sliding mode band and corresponding triggering conditions only considered the size of the error state. Maybe due to this fact, the global ETM that can ensure a unified lower bound for inter-event times globally is not fully investigated and designed. 

The global ETM for the results on practical sliding mode band was considered in few results (see \cite{7448882} and \cite{behera2021survey}) by designing the switching gain as a function of the sampled value of the state. While, it can lead to the result that the steady-state bound is
also determined by the sampled value of system state. Recently, \cite{behera2025event} provided a global ETM to effectively avoid the Zeno phenomenon globally by using digital redesign method. However, the event-triggered SMC in \cite{behera2025event} seems complicated and the control cost is increased by using additional digital redesign. Actually, the idea of the practical sliding mode band provides a valuable framework for event-triggered SMC by confining system trajectories within the vicinity of the sliding manifold, but it focuses exclusively on the magnitude of the error state. 
Namely, these results often ignored the direction information of error state, which may cause unnecessary resource loss. Specifically, there are cases where the system state, although not within the sliding mode band, does not have a significant deviation in terms of directional shift. Consequently, from the perspective of direction, there is no need to update the controller. The detailed discussions will be given in Section \uppercase\expandafter{3}.\\ 
\indent Although there are many significant results on event-triggered SMC, three key points are worth mentioning. First, the ETM structures designed in the existing related results only involve the size of error state, namely, its direction shift that can be helpful for the improvement of ETM is often ignored. Second, most of the previous literature only addresses the Zeno phenomenon in a semi-global scope, that is, the issue of positive lower bound is not well solved for inter-event time when the system state is infinite.
Third, the previous results on event-triggered SMC only ensure that the state of system can be ultimately bounded, rather than asymptotically converge to the equilibrium point.\\
\indent Motivated by the above-mentioned discussions, we are devoted to investigating the event-triggered SMC by using the new concept of ideal sliding mode cone. To be specific, to ensure global robust asymptotic stability of the linear time-invariant systems, we present a hybrid ETM, which includes the information of both the size and direction shift of error state. In addition, the proposed event-triggered SMC strategies can guarantee that systems trajectories are drawn towards the ideal sliding mode cone and asymptotically converge to the equilibrium point. The contributions of this article can be highlighted as follows:  

\begin{itemize}
	\item[(1)] Compared to previous results on event-triggered SMC based on practical sliding mode band (see \cite{7448882,behera2021survey,bandyopadhyay2018event,behera2025event,7048489}), we propose a new concept of the ideal sliding mode cone in this paper. Based on this concept, we design some event-triggered SMC laws with a hybrid ETM to achieve global robust stability of the addressed systems. 
	\item[(2)] The proposed hybrid ETM considers the information of both the size and direction shift of the error state, which can make energy utilization more efficient and ensure higher accuracy. It is shown that such an ETM is beneficial for the reduction of update frequency for control signals. Moreover, it is verified that the inter-event time possesses a positive lower bound if the state is relatively large and even infinite, which has not been well addressed in previous results (see  \cite{behera2021survey,bandyopadhyay2018event,7048489}).
	\item[(3)] Note that the existing results on practical sliding mode band can ensure that there is a lower bound on the inter-event time semi-globally, and the system state can only achieve ultimate boundedness (e.g., \cite{chen2025design,cucuzzella2020event,kumari2020event,liu2018event}). However, by using the concept of ideal sliding mode cone, we can verify that the inter-event time possesses a positive lower bound even if the state is infinite. Moreover, robust asymptotic stability of the addressed system can be achieved at the cost of potential frequent updating of control signals when the state is near the equilibrium point.
	\item[(4)] In order to balance the asymptotic convergence and the high-frequency triggering near the equilibrium point, the concept of ideal sliding mode cone is extended to the more practical case, i.e., practical sliding mode cone. Then, a global ETM is derived to guarantee a unified positive lower bound for all
	inter-event times without high-frequency triggering near the origin. To further demonstrate the feasibility of this idea, an application to the quadrotor unmanned aerial vehicle (UAV) systems is obtained.
\end{itemize}

The remaining portion is structured as follows: Section \uppercase\expandafter{2} proposes the model of the LTI systems with external disturbances, essential assumptions, and the related concepts. Section \uppercase\expandafter{3} presents the primary outcomes of the study, namely, some sufficient conditions for global robust stabilization of the addressed systems and Zeno-free phenomenon. Section \uppercase\expandafter{4} presents the concept of practical sliding mode cone to solve the lower bound existence issue for inter-event time near the equilibrium point. Section \uppercase\expandafter{5} applies the developed control framework to practical scenarios to show the effectiveness of the derived results. Finally, three numerical examples are given in Section \uppercase\expandafter{6} to illustrate the validity of our results. Section \uppercase\expandafter{7} summarizes the full study.

%The remainder of the paper is arranged as follows. In Section \uppercase\expandafter{2}, we formulate the problem, and necessary notations and definitions are given. In Section \uppercase\expandafter{3}, the main
%results are derived. Then, some examples are illustrated in Section \uppercase\expandafter{4}, and
%summaries are given in Section \uppercase\expandafter{5}.
\section{Preliminaries}
Let $\mathbb{Z}_+$ denote the set of all positive integers and $\mathbb{N}$ denote the set of all natural numbers. The sets of $n$-dimensional real vectors and $n\times m$-dimensional real matrices are denoted by $\mathbb{R}^{n}$ and $\mathbb{R}^{n\times m}$, respectively. The notation $A>0$ ($A<0$) indicates that $A$ is a positive (negative) definite, symmetric matrix.
The transpose and inverse of matrix $A\in \mathbb{R}^{n\times n}$ is denoted by $A^T$ and $A^{-1}$, respectively. Let $|\cdot|$ denote the absolute value function. The Euclidean norm of a vector $x\in \mathbb{R}^{n}$ is defined by $\|x\|=\sqrt{x^{T}x}$. The inner product \(u \cdot v\) is defined as the sum of the products of the corresponding entries of the vectors \(u\) and \(v\). The induced norm of matrix $A$ is defined as $\|A\|=\sqrt{\lambda_{\rm max}\{A^{T}A\}}$, where $\lambda_\textrm{max}\{A^{T}A\}$ denotes the maximum eigenvalue of the matrix $A^TA$.

We consider a linear system
\begin{equation}\label{LTI}
	\begin{split}
		\dot{\tilde{x}}(t)=\tilde{A}\tilde{x}(t)+\tilde{B}(u(t)+d(t)),
	\end{split}
\end{equation}
with initial condition $\tilde{x}(t_0)=\tilde{x}_0$, where $\tilde{x}(t)\in\mathbb{R}^n$, $u(t)\in\mathbb{R}$ are the system state and control input, respectively. It is assumed here that the external disturbance $d(t)\!\in\mathbb{R}$ is bounded by ${d}_\textrm{max}$, i.e., $\textrm{sup}_{t\ge{t}_0}\left|d\left(t\right)\right|\le{d}_\textrm{max}$, where ${d}_\textrm{max}$ is a known positive constant. And matrices  $\widetilde{A}\in\mathbb{R}^{n\times n}$ and $\widetilde{B}\in\mathbb{R}^{n\times 1}$ are constant matrices.\\
\indent Referring to previous literature \cite{bandyopadhyay2018event}, the following assumption is presented.
\begin{assumption}
	The pair $(\tilde{A},\tilde{B})$ is controllable.
\end{assumption}

At first, system \eqref{LTI} is transformed into regular form to design SMC. Actually, the system under proper nonsingular linear transformation $x=T_r\tilde{x}$ (see \cite{BEHERA201861,jiang2021event}) can be described by 
\begin{subequations}
	\begin{align}	
		\dot{x}_1(t)&={A}{_{11}}{x}_1(t)+{A}{_{12}}{x}_2(t)\tag{2a}\\
		\dot{x}_2(t)&={A}{_{21}}{x}_1(t)+{A}{_{22}}{x}_2(t)+u(t)+d(t)\tag{2b}
	\end{align}
\end{subequations}
where $x_1\in\mathbb{R}^{n-1}$ and $x_2\in\mathbb{R}$, and matrices $A_{ij},i,j\in \{1,2\}$ are of appropriate dimensions. System (2) can be rewritten in compact form as
\begin{equation}
	\dot{x}(t)=Ax(t)+B(u(t)+d(t)) \label{3}
\end{equation}
where  \[
A = \begin{bmatrix}
	A_{11} & A_{12} \\
	A_{21} & A_{22}
\end{bmatrix}
\] and $B=[0,0,...,1]^T\in\mathbb{R}^{n\times1}$.

Design the sliding variable 
\begin{equation}
	s=c^Tx
\end{equation} where $c\!\in\mathbb{R}^n$ such that $c\!=\!(c_1^T,1)^T$ with $c_1^T\!\in\!\mathbb{R}^{n-1}$. 

\noindent Therefore, the sliding manifold can be written by
\begin{equation*}
	S\triangleq\{x\in\mathbb{R}^n:s=c^Tx=0\}.
\end{equation*}
Here, ${c}_1$ is designed such that $A_{11}-A_{12}c_1^T$ is Hurwitz. According to Assumption 1, we can obtain that the pair $(A_{11},A_{12})$ is controllable \cite{edwards1998sliding}. Actually, due to the continuity of eigenvalues of matrix $A_{11}-A_{12}c_1^T$ concerning $c_1$, we can further find two sliding manifolds on either side of $S$.
To be specific, we define 
\begin{equation*}
	\hat{S}\triangleq\{x\in\mathbb{R}^n:\hat{s}=\hat{c}^Tx=0\}
\end{equation*}
\begin{flushleft}
	and
\end{flushleft}
\begin{equation*}
	\check{S}\triangleq\{x\in\mathbb{R}^n:\check{s}=\check{c}^Tx=0\}.	
\end{equation*}
Similarly, $\hat{c}_1$ and $\check{c}_1$ need to be selected such that both $A_{11}-A_{12}\hat{c}_1^T$ and $A_{11}-A_{12}\check{c}_1^T$ are Hurwitz.
\begin{remark}
	{\rm In particular, we take an example to illustrate the existence of sliding mode cone boundaries. Consider system \eqref{3} with $A=[4,\,6;-20,\,1]$, hence $A_{11}=4$ and $A_{12}=6$. We choose $c_{1}^{T}=3$, $\hat{c}_{1}^{T}=5$, and $\hat{c}_{1}^{T}=1$. The corresponding $1\times 1$-matrices are $A_{11}-A_{12}c_{1}^{T} = 4 - 6\times 3 = -14$, $A_{11}-A_{12}\hat{c}_{1}^{T} = 4 - 6\times5 = -26$, and $A_{11}-A_{12}\check{c}_{1}^{T} = 4 - 6\times1 = -2$. Since all of these are Hurwitz matrices, this verifies that sliding surfaces $\hat{S}$ (determined by $\hat{c}=[\hat{c}_{1}^{T}, 1]^T$) and $\check{S}$ (determined by $\check{c}=[\check{c}_{1}^{T}, 1]^T$) can indeed be designed on either side of the central sliding surface $S$ (determined by $c=[c_{1}^{T}, 1]^T$). More detailed discussions will be given in Section \uppercase\expandafter{4}}.
\end{remark}

Due to the fact that the ideal sliding mode cannot be acquired by using the event-triggered control strategy, many researchers turn to investigate the practical sliding mode band \cite{bandyopadhyay2018event}. The formal definition of practical sliding mode band is given below.
\begin{figure}[t]
	\centering
	\subfigure[Illustration of practical sliding mode band.]{\includegraphics[width=0.49\textwidth]{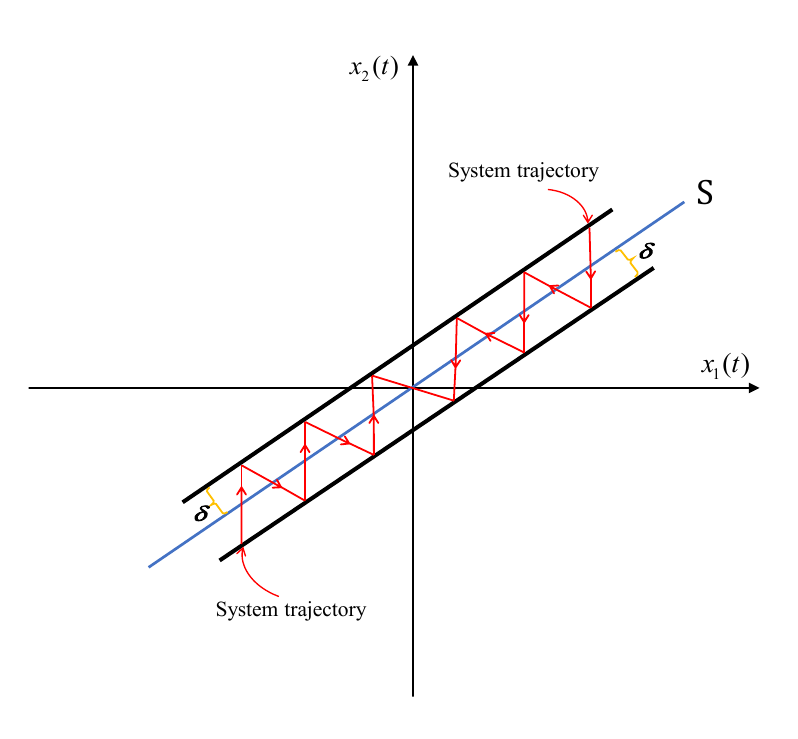}}
	\subfigure[Illustration of ideal sliding mode cone.]{\includegraphics[width=0.49\textwidth]{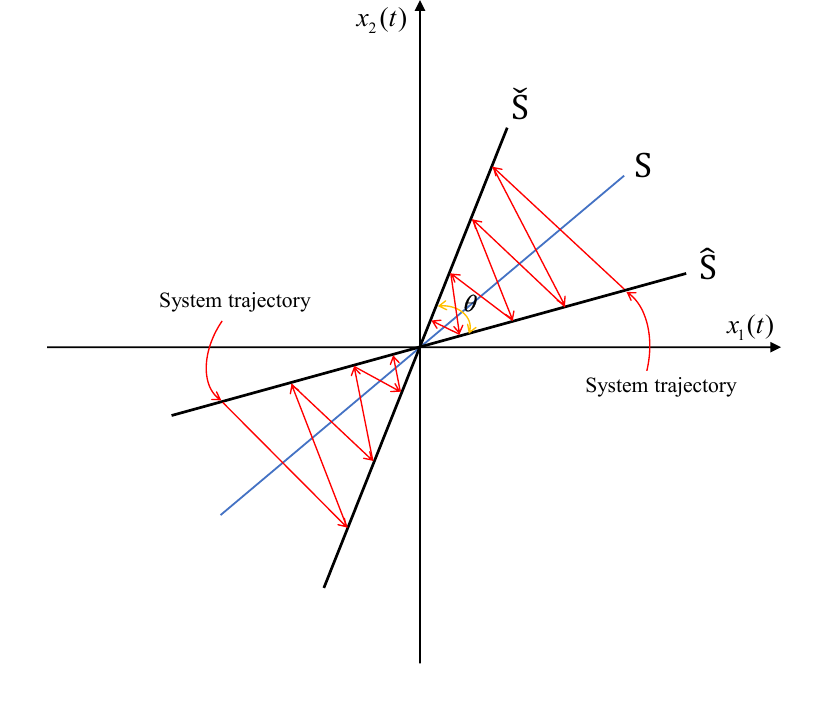}}
	\caption{Illustration of practical sliding mode band and ideal sliding mode cone.}
	\label{Fig. 1}
\end{figure}
\hfill	
\begin{definition}\label{PSMB}{\rm(Practical Sliding Mode Band  \cite{bandyopadhyay2018event}):}
	{\rm For any trajectory $x(t)$ of system \eqref{3} and a sliding manifold defined by $S$, system \eqref{3} is considered to be in practical sliding mode band if, for some positive $\delta>0$, there exists a time $t_1^*$ such that the sliding trajectories stay confined in the $\delta$-vicinity of sliding manifold for all time $t\ge t_1^*$. The bandwidth of the practical sliding mode band is given by $\delta$. In other words, $||s(t)||\le\delta$ for all $t\ge t_1^*$. In particular, the practical sliding mode is said to be ideal sliding mode if $\delta=0$.}
\end{definition}\\
\indent Actually, the practical sliding mode band can be obtained by the translation of the sliding mode surface $S$, and the ideal sliding mode cone that can be acquired through rotation of the sliding mode surface $S$ is investigated in this paper. Actually, it will be the key to ensure the global existence of positive lower bound for inter-event time. Specifically, the definition of ideal sliding mode cone is given as follows:
%\begin{figure}[t]
	%\centering
	%\includegraphics[width=3in]{huamozhui}
	%\caption{Illustration of ideal sliding mode cone.}\label{Fig. 2}
%\end{figure}  
\begin{definition}{\rm (Ideal Sliding Mode Cone):}
	{\rm System \eqref{3} is said to be in ideal  sliding mode cone if there exists a ${t}_1^*\ge{t}_0$ such that the system state reaches and remains within the cone region $D\triangleq\left\lbrace x\in\mathbb{R}^n:\hat{s}(t)\check{s}(t)\le0\right\rbrace$ for all $t\ge t_1^*$. In addition, the angle $\theta$ between the sliding manifolds $\hat{S}$ and $\check{S}$ is defined by $\theta=\pi -\text{arccos}\left(\frac{\check{c}^T\cdot\,\hat{c}}{\left\|\hat{c}\right\| \left\|\check{c}\right\|}\right)$.}
\end{definition}
\begin{remark}
	{\rm The concept of ideal sliding mode cone is fundamentally distinct from the practical sliding mode band utilized in previous studies. This distinction is visually represented in Fig.~\ref{Fig. 1}(a) and Fig.~\ref{Fig. 1}(b). Specially, Fig. \ref{Fig. 1}(a) illustrates the practical sliding mode band, which confines the trajectories of system within a band around the sliding surface (see \cite{behera2021survey,bandyopadhyay2018event,7048489}). In contrast, Fig. \ref{Fig. 1}(b) depicts the ideal sliding mode cone, which can be obtained by rotating the sliding manifold $S$. Moreover, we can observe that the sliding manifold $\hat{S}$ and $\check{S}$ sit on either side of $S$, and the angel between $\hat{S}$ and $\check{S}$ is equal to $\theta$. With the ideal sliding mode cone concept, one can ensure that the trajectories of system converge asymptotically to the equilibrium point, which is a substantial improvement over the practical sliding mode band methodology. Moreover, the inter-event time possesses a unified positive lower bound even if the system state is infinite.}
\end{remark}

Many previous results on the event-triggered SMC are based on the concept of practical sliding mode band (see \cite{7448882,behera2021survey,BEHERA201861}). However, it is worth noting that the majority of these results based on the practical sliding mode band concept have a limitation that the inter-event time has no positive lower bound when the state is infinite. Different from the previous results, the concept of ideal sliding mode cone is proposed in this study to overcome the above difficulty. The detailed discussions will be presented in Section \uppercase\expandafter{3}.

\section{Main results}
In this section, a hybrid ETM is introduced, which is helpful for the achievement of global robust asymptotic stability and some useful criteria are derived for the addressed closed-loop system. First, by referring to \cite{7448882}, the control law for event-triggered SMC is designed as follows:
\begin{equation}\label{7}
	u(t)=-({c}^TB)^{-1}\left({c}^TAx(t_i)+K\,(x(t_i))\,\textrm{sign}s({t}_i)\right) 
\end{equation}
for all $t\in\left\lbrack{t}_i,{t}_{i+1}\right), i\in \mathbb{N}$. In this case, the control input is maintained at this value until the next triggering instant occurs. The error state is defined by 
\begin{equation}\label{8}
	e(t)=x(t_i)-x(t)
\end{equation}
with $e(t_i)=0$ for $t\in\left\lbrack{t}_i,{t}_{i+1}\right), i\in \mathbb{N}$.
Especially, the triggering instant $t_{i+1}\!=\textrm{min}(t_{i+1}^{\prime},t_{i+1}^{\prime\prime})$, $i\in\mathbb{N}$ is generated by the hybrid ETM:
\begin{equation}
	t_{i+1}^{\prime}=\!\textrm{inf}\left\lbrace t:\!t>t_i^{\prime},\frac{x^{T}(t)\cdot
		x(t_i)}{||x(t)||\> ||x({t}_{i})||}\le \textrm{cos}\frac{\theta}{2n}\right\rbrace \label{5}
\end{equation}
\begin{equation}
	t_{i+1}^{\prime\prime}=\!\textrm{inf}\left\lbrace t:\!t>{t}_i^{\prime\prime},|c^TAe(t)|\ge\sigma||x({t}_i)||+\beta\right\rbrace \label{6}
\end{equation}
where ${\sigma}$ and $\beta$ are two positive constants and $n\in\mathbb{Z}_+$. In addition, $\theta$ represents the magnitude of angle between the sliding manifolds $\hat{S}$ and $\check{S}$. Due to the fact that system \eqref{3} under control law \eqref{7} is characterized by a discontinuous dynamical system, the solutions to  system \eqref{3} are understood in the sense of Filippov \cite{filippov1988differential}.
\begin{remark}
	{\rm It is worth mentioning that the above hybrid ETM contains both  size (i.e., \eqref{5}) and direction shift (i.e., \eqref{6}) of the error state, which is different from the existing results. It will be shown that by using this hybrid ETM, the inter-event time can possess a unified positive lower bound regardless of the boundedness of the system state. In addition, the parameter $n$ in triggering rule \eqref{5} and the state-dependent control gain $K(x(t_i))$  are to be selected. Actually, if both of them are sufficiently large, control law \eqref{7} with hybrid ETM \eqref{5}-\eqref{6} can ensure that the system state reaches and remains in the ideal sliding mode cone in finite time. Some detailed discussions can be seen in Example 1 of Section \uppercase\expandafter{6}.} 
\end{remark}
\begin{remark}
	{\rm Clearly, the allowed direction deviation of system state (i.e., $\frac{\theta}{2n}$ in \eqref{5}) will decrease if $n$ increases. In fact, larger $n$ may lead to higher control frequency, but ensure better convergence. In addition, if given $K(x(t_i))$, increasing $\sigma$ or $\beta$ in ETM \eqref{6} generally can relax the number of triggers (larger $T_{i2}$ in \eqref{15b}), particularly when the state norm is relatively large.}
\end{remark}
\begin{theorem}\label{thm-bound}
{\rm Consider system \eqref{3} under event-triggered SMC \eqref{7} with hybrid ETM \eqref{5}-\eqref{6}. If $K(x({t}_i))$ is selected such that
\begin{equation}\label{9}
	K(x({t}_i))>{\left|{c}^TB\right|d}_\textrm{max}+\sigma||x(t_i)||+\beta,
\end{equation}
then system \eqref{3} can be enforced in the ideal sliding mode cone region, and global robust asymptotic stability of the closed-loop system can be achieved.}
\end{theorem}
\begin{proof}
	Consider Lyapunov function 
	\begin{equation*}
		V=\frac{1}{2}{s}^2.
	\end{equation*} 
	First, we shall show that for all time $t\in\left\lbrack{t}_i,{t}_{i+1}\right)$ and $i\in\mathbb{N}$, the state trajectories of the closed-loop system can be brought to the vicinity of sliding manifold $S$. Actually, taking the time derivative of $V$ along the trajectories of the addressed system, one can deduce that 
	\begin{equation}\label{10}
		\begin{aligned}
			\dot{V}&=s\cdot\dot{s}\\
			&=s(c^{T}Ax(t)+c^{T}Bu(t)+c^{T}Bd(t))\\
			&=s({c}^TAx(t)\!-\!{c}^TAx({t}_i)\\
			&\quad-K(x({t}_i))\,\textrm{sign}s({t}_i)\!+\!{c}^TBd(t)).
		\end{aligned}
	\end{equation}
	Recall $e\left(t\right)=x({t}_i)-x(t)$
	and \eqref{10} can be further deduced:
	\begin{equation}
		\begin{aligned}
			\dot{V}&=-s(t){c}^TAe(t)-s(t)\,K(x({t}_i))\,\textrm{sign}s({t}_i)+s(t){c}^TBd(t)\\
			&\le-s(t)\,K(x({t}_i))\,\textrm{sign}s({t}_i)+|s(t)|\,|c^{T}Ae(t)|\\
			&\quad+|s(t)|\,|c^{T}Bd(t)|.
		\end{aligned}
	\end{equation}
	When the system trajectories start from a domain where $\textrm{sign}s({t}_i)=\textrm{sign}s(t)$, the aforementioned relationship can be further simplified from \eqref{6} and \eqref{9} that 
	\begin{equation*}
		\begin{aligned}
			\dot{V}&\leq-|s(t)|\,K(x(t_i))+|s(t)|\,|c^TAe(t)|+|s(t)|\,|c^TBd(t)|\\
			&\le-|s(t)|K(x(t_i))+|s(t)|\sigma||x(t_i)||+|s(t)|\beta\\
			&\quad+|s(t)|\,|c^TB|d_\textrm{max}\\
			&=-|s(t)|\left(K(x(t_i))-\sigma||x(t_i)||-\beta-|c^TB|d_\textrm{max}\right)\\
			&\le-\xi|s(t)|
		\end{aligned}
	\end{equation*}
	where $\xi>0$ such that $K(x({t}_i))\ge\sigma||x(t_i)||+\beta+|c^TB|d_\textrm{max}+\xi$. We can observe that the trajectories of system are attracted to the sliding manifold as long as $\textrm{sign}s(t_{i})=\textrm{sign}s(t)$. Because $\dot{V}<0$, the trajectories eventually reach the sliding manifold $S$ at some instant $t_1^*$.\\
	\indent It is not difficult to see that when the state trajectories reach the sliding mode surface $S$, it is inside the ideal sliding mode cone. It shall be proved that once the state trajectories reach the sliding mode surface $S$, it will subsequently remain inside the ideal sliding mode cone. By referring to \cite{7448882}, we only need to consider the worst case, which is $s(t_{i})=0$. Since the angular size of the ideal sliding mode cone is $\theta$, and the sliding mode surface $S$ is exactly in the center of the ideal sliding mode cone, it can be seen from \eqref{5} that the state trajectories do not leave the ideal sliding mode cone before next triggering instant. Actually, when parameter $n$ is chosen to be sufficiently large, the angle shift allowed within each inter-event interval is small enough. Choosing proper parameter $n$ can ensure that the state remains in the ideal sliding mode cone for the next few inter-event intervals. In addition, in this process, the control law \eqref{7} has the effect of making the state reach sliding mode surface $S$ again, i.e., the center of the ideal sliding mode cone. Therefore, the system state will always remain in the ideal sliding mode cone. \\
	\indent Second, we shall prove that the global robust asymptotic stability of the addressed system can be derived. When the system is contained within the region of ideal sliding mode cone, there must be two time-varying parameters $\lambda_1(t)$ and ${\lambda}_2(t)$, satisfying ${\lambda}_1(t), {\lambda}_2(t)\in\left\lbrack0,1\right\rbrack$ and ${\lambda}_1(t)+{\lambda}_2(t)\equiv1$ such that
	\begin{equation*}
		\lambda_1(t)\,\hat{s}(t)+\lambda_2(t)\,\check{s}(t)\equiv0.
	\end{equation*}
	That is to say, we can obtain that
	\begin{equation*}
		\lambda_1(t)\,[\hat{c}_1^{T}x_{1}+x_{2}]+\lambda_2(t)\,[\check{c}_1^{T}x_{1}+x_{2}]\equiv0.
	\end{equation*} 
	Due to the fact that the state variable $x_2(t)$ can be written by $x_2(t)\!\!=\!\!-\lambda_1(t)\hat{c}_1^{T}x_{1}-{\lambda}_2(t)\check{c}_1^{T}x_{1}$, the reduced order system dynamics can be described by
	\begin{equation}
		\dot{x}_{1}={A}_{11}x_{1}+{A}_{12}[-{\lambda}_1(t)\hat{c}_1^{T}x_{1}-{\lambda}_2(t)\check{c}_1^{T}x_{1}].
	\end{equation}
	\indent Next, consider another Lyapunov function $V_1=x_{1}^{T}Px_{1}$, where $P$, $\hat{Q}$ and $\check{Q}$ are positive-define matrices that satisfy $({A}_{11}-{A}_{12}\hat{c}_1^T)^TP+P({A}_{11}-{A}_{12}\hat{c}_1^{T})=-\hat{Q}$ and $({A}_{11}-{A}_{12}\check{c}_1^{T})^TP+P({A}_{11}-{A}_{12}\check{c}_1^{T})=-\check{Q}$. Due to the fact that both $A_{11}-A_{12}\hat{c}_1^T$ and $A_{11}-A_{12}\check{c}_1^T$ are Hurwitz matrices, these two Lyapunov equations can be fulfilled for some positive definite matrices $\hat{Q}$ and $\check{Q}$. Differentiating $V_1$ with respect to time, it yields that 
	\begin{equation}\label{13}
		\begin{aligned}
			\dot{V}_1&={\dot{x}_{1}}^{T}Px_{1}+x_{1}^{T}P\dot{x}_{1} \\
			&={[{A}_{11}x_{1}+{A}_{12}(-{\lambda}_1(t)\hat{c}_1^Tx_{1}-{\lambda}_2(t)\check{c}_1^Tx_1)]}^{T}Px_{1}\\
			&\quad+x_1^TP[{A}_{11}x_{1}+{A}_{12}(-{\lambda}_1(t)\hat{c}_1^{T}x_{1}-{\lambda}_2(t)\check{c}_1^{T}x_{1})].
		\end{aligned}
	\end{equation}
	Since $\lambda_1(t)+\lambda_2(t)\equiv1$, \eqref{13} can be further deduced that
	\begin{equation}
		\begin{aligned}
			\dot{V}_1&=(\lambda_1(t)+\lambda_2(t))x_1^TA_{11}^TPx_1-\lambda_1(t)x_1^T(A_{12}\hat{c}_1^T)^TPx_1\\
			&\quad-\lambda_2(t)x_1^T(A_{12}\check{c}_1^T)^TPx_1+(\lambda_1(t)+\lambda_2(t))x_1^TPA_{11}x_1\\
			&\quad-\lambda_1(t)x_1^TPA_{12}\hat{c}_1^Tx_1-\lambda_2(t)x_1^TPA_{12}\check{c}_1^Tx_1\\
			&=\lambda_1(t)x_1^T[({A}_{11}-{A}_{12}\hat{c}_1^T)^TP+P({A}_{11}-{A}_{12}\hat{c}_1^{T})]x_1\\
			&\quad+\lambda_2(t)x_1^T[({A}_{11}-{A}_{12}\check{c}_1^{T})^TP+P({A}_{11}-{A}_{12}\check{c}_1^{T})]x_1\\
			&=-{\lambda}_1(t)x_1^T\hat{Q}x_1-{\lambda}_2(t)x_1^T\check{Q}x_1.
		\end{aligned}
	\end{equation}
	Hence, derivative of $V_1$ is negative definite. It implies that the state variable $x_1(t)$ of system \eqref{3} eventually converges to the zero vector. On the other hand, using relation $x_2(t)\!=\!-{\lambda}_1(t)\,\hat{c}_1^Tx_1-{\lambda}_2(t)\,\check{c}_1^Tx_1$, one can obtain that
	\begin{equation*}
		\begin{aligned}
			||x(t)||&\le||x_1(t)||+||x_2(t)||\\
			&=||x_1(t)||\,(1+{||\lambda}_1(t)\hat{c}_1^T||+||\lambda_2(t)\check{c}_1^T||)\\
			&\leq ||x_1(t)||(1+||\hat{c}_1^T||+||\check{c}_1^T||).
		\end{aligned}
	\end{equation*}
	Therefore, based on the above analysis, we can conclude that system \eqref{3} is globally robustly asymptotically stable. The proof is completed.
\end{proof}

\begin{remark}\rm
	In Theorem \ref{thm-bound}, we have demonstrated that the global robust asymptotic stability of the system is achieved through the proposed  event-triggered SMC \eqref{7} with hybrid ETM \eqref{5}-\eqref{6}. Recently, robust asymptotic stability was also obtained for event-triggered SMC systems by using a well-designed dynamic event-triggered approach and observer-based sliding mode surface (see \cite{liu2020event}). Note that the event-triggered SMC in \cite{liu2020event} is not kept by classical zero-order holder, and the control laws are designed by additional state observer, which is helpful for enforcing states into the ideal sliding mode. In fact, it may increase the control cost and  computational complexity. In addition, regarding the literature on event-triggered SMC based on zero-order holders, such systems usually cannot reach the ideal sliding mode because of its discretization nature. From a technical perspective, if the states reach the practical sliding mode band instead of the ideal sliding mode, we can only ensure ultimate boundedness rather than asymptotic stability (see \cite{chen2025design,cucuzzella2020event,kumari2020event}).
\end{remark}

In addition, the time sequence ${\left\lbrace{t}_{i}\right\rbrace}$ generated by the proposed hybrid ETM \eqref{5}-\eqref{6} is global admissible, which is shown in the following result.
\begin{theorem}\label{thm2-bound}
	{\rm Consider system \eqref{3} under event-triggered SMC \eqref{7} with hybrid ETM \eqref{5}-\eqref{6}. Then, there exists a unified minimum positive inter-event time $T_i=\textrm{min}(T_{i1},T_{i2}), i\in \mathbb{N}$
		where $T_{i1}$ and $T_{i2}$ satisfy
		\begin{subequations}
			\begin{align}
				T_{i1}&\ge\frac{1}{||A||}\,\textrm{ln}\left(1+\frac{\textrm{sin}\frac{\theta}{2n}}{\rho}||A||-\frac{\frac{\gamma}{\rho}||A||}{\rho\,||x(t_{i})||+\gamma}\right)\label{15a}\\ T_{i2}&\ge\frac{1}{||A||}\,\textrm{ln}\left(1+\frac{\sigma||x(t_i)||+\beta}{||c||\>\rho(||x(t_i)||)+\gamma}\right)\label{15b},
			\end{align}
		\end{subequations}
		and $\rho$ as well as $\gamma$ are defined by 
		\begin{equation}
			\begin{aligned}
				\rho&\triangleq||A-B(c^TB)^{-1}c^{T}A||\\
				\gamma&\triangleq||B(c^TB)^{-1}||\>||K(x(t_i))||+||B||d_\textrm{max}.
			\end{aligned}
		\end{equation}
	}
\end{theorem}

\begin{proof}
	Consider set $\Gamma\!=\!\lbrace t\in(t_i,\infty):||e(t)||\!=\!0\rbrace$. For all $t\in[t_i,t_{i+1})\setminus\Gamma$, it implies that
	\begin{equation}
		\begin{aligned}
			\frac{d}{dt}||e(t)||&\le\left|\left|\frac{d}{dt}e(t)\right|\right|=\left|\left|\frac{d}{dt}x(t)\right|\right|\\
			&={\bigg|\bigg|Ax(t)-B(c^TB)^{-1}c^TAx(t_i)\Bigg.}\\ \notag&\quad-{\bigg.B(c^TB)^{-1}K(x(t_i))\,\textrm{sign}s(t_i)+Bd(t)\bigg|\bigg|}.
		\end{aligned}
	\end{equation}
	Using the relation $e(t)\!=\!x(t_i)\!-\!x(t)$ in the above discussion and simplifying further, it yields
	\begin{equation*}
		\begin{aligned}
			\frac{d}{dt}||e(t)||&\le||A||\>||e(t)||+\left|\left|A-B(c^TB)^{-1}c^TA\right|\right|||x(t_i)||\\
			&\quad+||B(c^TB)^{-1}||\>||K(x(t_i))||+||B||d_\textrm{max}\\
			&=||A||\>||e(t)||+\rho||x(t_i)||+\gamma.
		\end{aligned}
	\end{equation*}
	Thus, by using comparison lemma \cite{khalil2002nonlinear} and the initial condition $||e(t_i)||=0$, we have
	\begin{equation}\label{upperboundfore}
		||e(t)||\le\frac{\rho||x(t_i)||+\gamma}{||A||}{\left(e^{||A||(t-t_i)}-1\right)}
	\end{equation}
	for $t\in[t_i,t_{i+1}), i\in \mathbb{N}$.  According to triggering rule \eqref{5}, one can obtain the time $T_{i1}$ needed for $||e(t)||$ to increase from $0$ to $\sin \frac{\theta}{2n} ||x(t_i)||$: 
	\begin{equation}\label{18}
		\textrm{sin}\frac{\theta}{2n}\,||x(t_i)||\le\frac{\rho||x(t_{i})||+\gamma}{||A||}\left({e}^{||A||T_{i1}}-1\right).
	\end{equation}
	Inequality \eqref{18} provides the expression for inter-event time $T_{i1}$, as given in \eqref{15a}.\\
	\indent Next, similar to proving \eqref{15a}, we can prove that \eqref{15b} is satisfied for the inter-event time generated by triggering rule \eqref{6}. For all $t\in[t_i,t_{i+1})\setminus\Gamma$, we can also derive \eqref{upperboundfore}.
	According to triggering rule \eqref{6}, one can get that the time $T_{i2}$ satisfies
	\begin{equation}
		\begin{split}
			\frac{\sigma||x(t_{i})||+\beta}{||c||\>||A||}\le\frac{\rho||x(t_{i})||+\gamma}{||A||}\left({e}^{||A||T_{i2}}-1\right)
		\end{split}
	\end{equation}
	which leads to \eqref{15b}. It is clear that the right-hand sides of both \eqref{15a} and \eqref{15b} possess a positive lower bound even if the system state is infinite. Hence, the proof is completed.
\end{proof}
%\begin{figure}[!t]
	%\centering
	%\includegraphics[width=3.5in]{diagram}
	%\caption{Block diagram of the event-triggered SMC based on the concept of practical sliding mode cone.}
	%\label{Fig. 3}
%\end{figure}
\begin{remark}
	{\rm In this study, the inter-event time can be guaranteed to have a unified lower bound even if the state is relatively large and even infinite. As exemplified by \cite{bandyopadhyay2018event}, while a positive lower bound on the inter-event time can often be established when the system state resides within a bounded region, achieving a uniform lower bound for all possible system states, particularly when the state is infinite, remains an unresolved issue.
Further, the system states are proved to be robustly asymptotically stable instead of just converging to a certain region of the equilibrium point, which is different from the previous results based on practical sliding mode band (see \cite{7448882,BEHERA201861,jiang2024periodic}). On the other hand, we can observe from \eqref{15a} that the global robust asymptotic stability can be achieved at the potential cost of high-frequency updating of control signals when the system states are close to the equilibrium point. In next section, we are going to address this potential cost by combining ideal sliding mode cone and practical sliding mode band.} 
\end{remark}\\
\indent \!\!In Theorem \ref{thm-bound}, one may observe that the control law \eqref{7} is designed to overcome the difficulty that the positive lower bound of inter-event time may not be satisfied before the system trajectories enter the ideal sliding mode cone (see \cite{7448882}). Actually, after the system trajectories enter the ideal sliding mode cone, the control law \eqref{7} can be replaced by the classical event-triggered SMC:
\begin{equation}\label{21}
	u=-(c^TB)^{-1}(c^TAx(t_i)+K\textrm{sign}s(t_i)).
\end{equation}
Once the system state enters the ideal sliding mode cone, its dynamic characteristics are constrained within the cone, and it only needs to maintain the system state within the cone to converge to the zero vector (as proven by Theorem \ref{thm-bound}). Continuing to use the control law \eqref{7} at this point will increase unnecessary computational burden and control costs.
Hence, we can further obtain the following results.
\begin{theorem}\label{thm3-bound}
	{\rm Consider system \eqref{3}. When the system state is outside the ideal sliding mode cone $D$, apply triggering rule \eqref{6} and control law \eqref{7} to drive the state toward $D$. Once the state enters $D$, the control inputs switch to law \eqref{21}.  The system trajectories are therefore guaranteed to stay within the ideal sliding mode cone by control law \eqref{21} with the triggering rule \eqref{5} and global robust asymptotic stability can be achieved.}
\end{theorem}
\begin{figure}[t]
	\centering
	\includegraphics[width=3.9in]{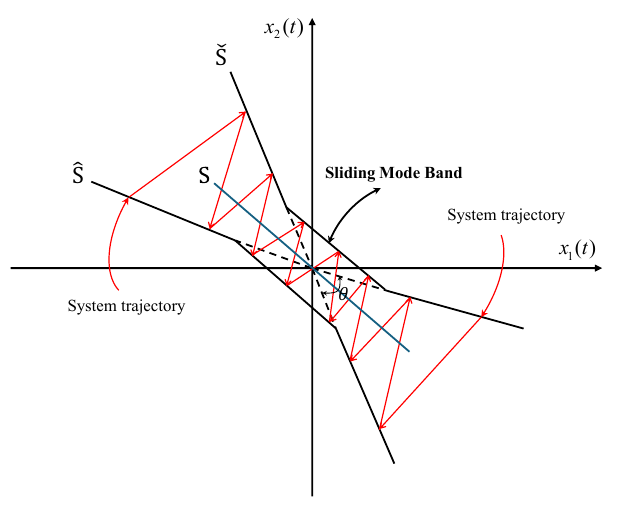}
	\caption{Illustration of practical sliding mode cone.}\label{Fig. 3}
\end{figure}
\begin{proof}
	The proof to this theorem shall be shown in the following three steps.\\
	\indent Step (i): By referring to the proof to Theorem \ref{thm-bound}, it can be obtained that under event-triggered SMC \eqref{7} with triggering rule \eqref{6}, it leads to
	\begin{equation*}
		\dot{V}\le-\xi|s(t)|
	\end{equation*} where $\xi>0$. The derivative of Lyapunov function $V$ is negative, so the control law \eqref{7} will drive the system towards the ideal sliding mode cone.\\
	\indent Step (ii): When the system state reaches the ideal sliding mode cone, control law \eqref{7} switches to \eqref{21}, and the triggering rule \eqref{5} begins to play a dominate role. Actually, once the state direction is significantly shifted, controller \eqref{21} will be updated to keep the system state in the ideal sliding mode cone.\\
	\indent Step (iii): Similarly, by referring the second half of proof to Theorem \ref{thm-bound}, we can observe that the addressed system can achieve global robust asymptotic stability.
\end{proof}
\begin{figure}[!t]
	\centering
	\includegraphics[width=4.2in,height=2.6in]{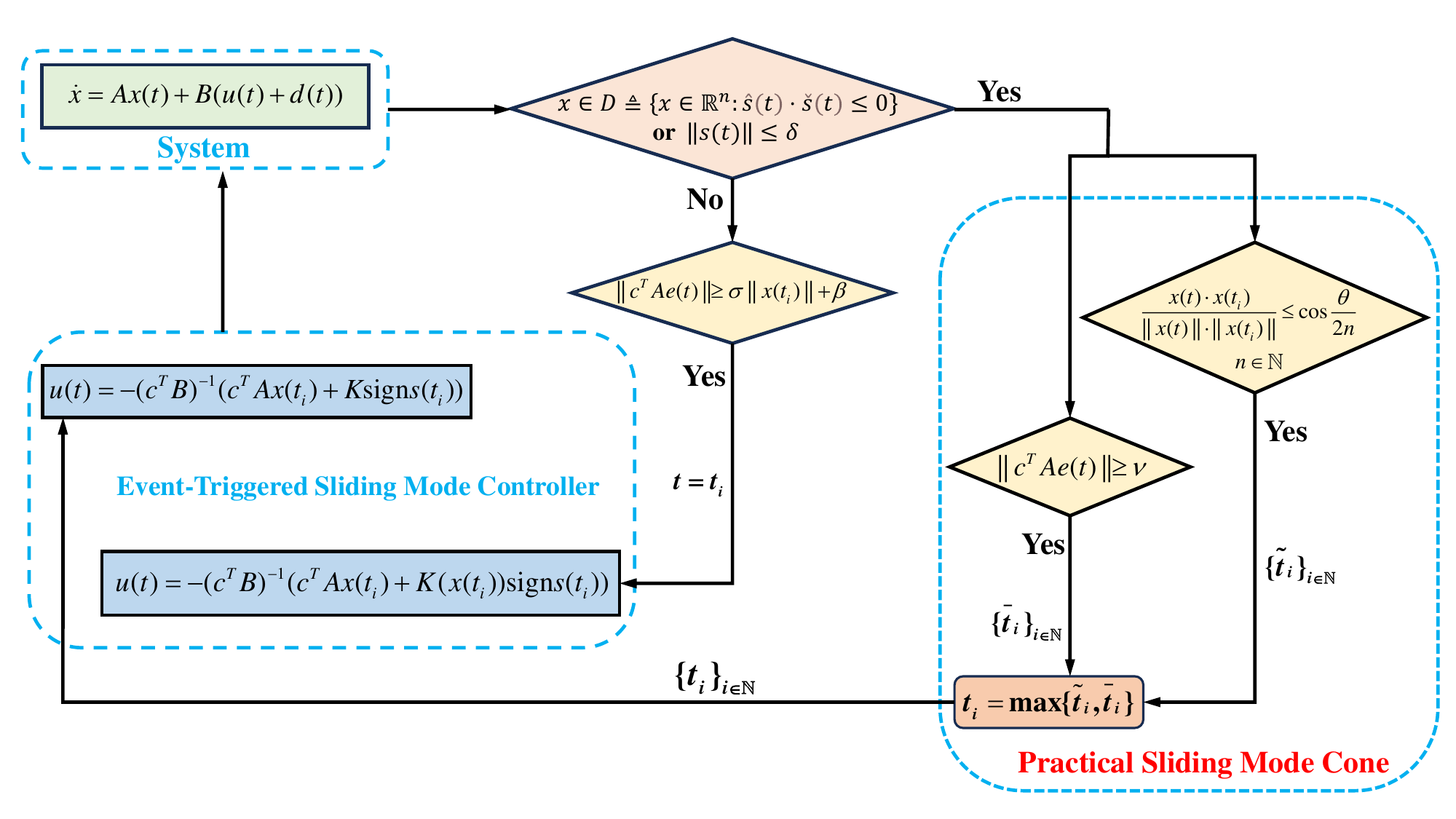}
	\caption{Block diagram of the event-triggered SMC based on the practical sliding mode cone.}\label{Fig. 4}
\end{figure}

Similarly, we can present the following result for the global admission of triggering time sequence.
\begin{theorem}\label{thm4-bound}
	{\rm Consider system \eqref{3} under the event-triggered SMC with hybrid ETM described by Theorem \ref{thm3-bound}. Then, the inter-event time
		can possess a unified positive lower bound even if the system states are infinite.}
\end{theorem}
\begin{proof}
As the proof of this theorem is similar to that of Theorem \ref{thm2-bound}, it is not repeated here.
\end{proof}\\

\section{Extension}
\indent Note that the concept of ideal sliding mode cone can be applied to guarantee that robust asymptotic stability is achieved and the inter-event time has a unified positive lower bound when the initial states of system are relatively large and even infinite. However, it may bring high-frequent updating of control inputs when the states are near the equilibrium point, which is the price paid for suppressing the external disturbances. It is known that the practical sliding mode band
can be regarded as a extension of ideal sliding mode (see Definition \ref{PSMB} and  \cite{behera2021survey,bandyopadhyay2018event,7048489}).  To relax the high-frequent updating near the origin and combine the advantages of both ideal sliding mode cone and practical sliding mode band, this section shall extend the concept of ideal sliding mode cone to the practical one. Specifically, the definition of practical sliding mode cone is given as follows: 
\begin{definition}{\rm (Practical Sliding Mode Cone):}
	{\rm System \eqref{3} is said to be in the practical conical sliding mode if there exists a $t_1^\star\geq t_0$ such that the system state reaches and remains within the following practical cone region $$D^\star\triangleq\left\lbrace x\in\mathbb{R}^n:\hat{s}(t)\check{s}(t)\le0~{\rm or}~\|s(t)\|\leq\delta \right\rbrace$$ for all $t\geq t_1^\star$.} 
\end{definition}

The practical sliding mode cone is illustrated in Fig.~\ref{Fig. 3}, showing the boundaries within which the system state is maintained.  Considering $n\in \mathbb{Z}_+$, the hybrid ETM is designed for practical sliding mode cone as follows:
\begin{equation}\label{22}
	\tilde{t}_{i+1}=\!\textrm{inf}\left\lbrace t:\!t>\tilde{t}_i,\frac{x^{T}(t)\cdot x(t_i)}{||x(t)||\> ||x({t}_{i})||}\le \textrm{cos}\frac{\theta}{2n}\right\rbrace
\end{equation}
\begin{equation}\label{23}
	\overline{t}_{i+1}=\!\textrm{inf}\left\lbrace
	t:\!t>\overline{t}_i,||c^TAe(t)||\ge\nu\right\rbrace
\end{equation} for $\nu\in(0,\infty)$, where $\theta$ represents the magnitude of angle between the sliding manifolds $\hat{S}$ and $\check{S}$. Based on the above hybrid ETM, the triggering instant is generated by $t_{i+1}=\textrm{max}(\tilde{t}_{i+1},\overline{t}_{i+1}), i\in \mathbb{N}$.\\
\indent The control algorithm based on the practical sliding mode cone for system \eqref{3} is shown in Fig. \ref{Fig. 4}. To be specific,\\
\indent 1) When the state trajectory of system is outside the ideal sliding mode cone, we consider the control laws \eqref{7} with triggering rule \eqref{6}. \\
\indent 2) After the state trajectory of system enters the practical sliding mode cone, the control laws will switch to \eqref{21} with ETM \eqref{22}-\eqref{23}.

Based on the above discussions, we can obtain the following result:
\begin{theorem}\label{thm5-bound}
	{\rm The control algorithm described by Fig. \ref{Fig. 4} can bring system \eqref{3} to the practical sliding mode cone. Specially, the system remains ultimately bounded in the following region:
		\begin{equation*}
			\Omega\triangleq\left\{x\in\mathbb{R}^{n}:||x||\le\frac{\nu}{||A||}+\frac{2\nu(1+||c_1||)||\tilde{P}A_{12}||}{\lambda_{\rm min}\{\tilde{Q}\}||A||}\right\}.
	\end{equation*}}
\end{theorem}
\begin{proof}
	First, we can observe that the state trajectory of system can be attracted into the ideal sliding mode cone under SMC \eqref{7} with triggering rule \eqref{6}. Moreover, the system state can maintain in practical sliding mode cone subsequently. Actually, the proof of this case is similar to that of Theorem \ref{thm3-bound}, so it is omitted here.\\
	\indent Second, we shall demonstrate that the system is ultimately bounded after arriving in the practical sliding mode cone. Especially, if system state is kept in region $D\triangleq\left\lbrace x\in\mathbb{R}^n:\hat{s}(t)\check{s}(t)\le0\right\rbrace$ all the time, we can conclude the asymptotic stability by using Theorem \ref{thm-bound}. Recall that the controller is changed to \eqref{21} with ETM \eqref{22}-\eqref{23} when the system state stay in the practical sliding mode cone. Then, from the other case that system state retain in region $\|s(t)\|\leq\delta$, it follows that 
	\begin{equation}\label{24}
		\begin{aligned}
			|s(t_i)-s(t)|&=|c^Tx(t_i)-c^Tx(t)|\\
			&\le||c||\>||e(t)||\\
			&\le\frac{\nu}{||A||}.
		\end{aligned}
	\end{equation}
	This gives the maximum deviation of sliding variable from its immediate sampled value. Then, the maximum value of the band can be obtained for the case $s(t_i)=0$.\\
	\indent To prove the uniform ultimate boundedness, consider the Lyapunov function $\tilde{V}_1(x_1)=x_1^T\tilde{P}x_1$ with positive definite matrix $\tilde{P}\in\mathbb{R}^{(n-1)\times(n-1)}$ such that $(A_{11}-A_{12}c_1^T)^T\tilde{P}+\tilde{P}(A_{11}-A_{12}c_1^T)=-\tilde{Q}$, for any positive definite matrix $\tilde{Q}\in\mathbb{R}^{(n-1)\times(n-1)}$. Differentiating $\tilde{V}_1$ with respect to time, it yields that
	\begin{equation}
		\begin{aligned}	
			\dot{\tilde{V}}_1(x_1)&=\dot{x}_1^T\tilde{P}x_1+x_1^T\tilde{P}\dot{x}_1\\
			&=x_1^T[(A_{11}-A_{12}c_1^T)^T\tilde{P}+\tilde{P}(A_{11}-A_{12}c_1^T)]x_1\\
			&\quad+2x_1^T\tilde{P}A_{12}s(t)\\
			&\le-x_1^T\tilde{Q}x_1+2||x_1||\>||\tilde{P}A_{12}||\>||s(t)||
		\end{aligned}
	\end{equation}
	where $A_{11}-A_{12}c_1^T$ is designed as a Hurwitz matrix. Since sliding variable $s(t)$ satisfies \eqref{24}, one can obtain that
	\begin{equation*}
		\dot{\tilde{V}}_1(x_1)\le-\lambda_{\rm min}\{\tilde{Q}\}||x_1||^2+\frac{2\nu||x_1||\>||\tilde{P}A_{12}||}{||A||},
	\end{equation*} 
	where $\lambda_{\rm min}\{\tilde{Q}\}$ denotes the minimum eigenvalue of $\tilde{Q}$.
	Hence, $x_1(t)$ remains bounded, namely, $||x_1||\le2\nu\frac{||\tilde{P}A_{12}||}{\lambda_{\rm min}\{\tilde{Q}\}||A||}$. Since $||s(t)||$ and $x_1(t)$ is bounded, we can acquire the boundedness of $x_2(t)$. To be specific,
	\begin{equation*}
		\begin{aligned}
			||x_2(t)||&\le||s(t)-c_1^Tx_1(t)||\\
			&\leq||s(t)||+||c_1||\>||x_1(t)||\\
			&\leq \frac{\nu}{||A||}+\frac{2\nu||c_1||\>||\tilde{P}A_{12}||}{\lambda_{\rm min}\{\tilde{Q}\}||A||}.
		\end{aligned}
	\end{equation*}
	Consequently, the system is uniformly ultimately bounded within the region $\Omega$. The proof is completed. 
\end{proof}
\begin{remark}
	{\rm In Section \uppercase\expandafter{3}, we have demonstrated that the trajectory of system \eqref{3} can always remain inside the ideal sliding mode cone and eventually converge to the equilibrium point under event-triggered SMC \eqref{7} with the proposed ETM, but we may not ensure that there is a positive lower bound on the inter-event time of system \eqref{3} near the equilibrium point. As a result, we propose the concept of practical sliding mode cone, which contains the advantages of both practical sliding mode band and ideal sliding mode cone. \\
		\indent Based on the practical sliding mode cone, the trajectory of the system remains bounded within a practical sliding mode band near the equilibrium point. Nevertheless, when the initial value is relatively large or even infinite, the lower bound of the inter-event time is still guaranteed by the sliding mode cone method. In other words, if only the practical sliding mode band is used without applying the practical sliding mode cone, then the global ETM cannot be derived. Next, we shall demonstrate that system \eqref{3} under event-triggered SMC \eqref{21} with ETM \eqref{22}-\eqref{23} can ensure that the inter-event time can globally possess a unified lower bound without high-frequency triggering near the origin.\\
		\indent Additionally, Fig.~\ref{Fig. 4} presents a block diagram clearly depicting the interplay between system states, control inputs, and the ETM. This figure illustrates how control updates are triggered based on the current system state and the predefined practical sliding mode cone, effectively reducing the frequency of control inputs.}
\end{remark}
\begin{theorem}\label{thm6-bound}
	{\rm Consider system \eqref{3} under the event-triggered SMC strategy described by Fig. \ref{Fig. 4}. Let ${\left\lbrace{t}_{i}\right\rbrace}_{i=0}^{\infty}$ be the sequence of triggering instants generated by ETM \eqref{22}-\eqref{23}, then there exists a maximum positive inter-event time $T_i=\textrm{max}(\tilde{T}_{i1},\overline{T}_{i2}), i\in \mathbb{N}$
		where $\tilde{T}_{i1}$ and $\overline{T}_{i2}$ satisfy 
		\begin{subequations}
			\begin{align}
				\tilde{T}_{i1}&\ge\frac{1}{||A||}\,\textrm{ln}\left(1+\frac{\textrm{sin}\frac{\theta}{2n}}{\rho}||A||-\frac{\frac{\gamma}{\rho}||A||}{\rho\,||x(t_{i})||+\gamma}\right)\label{27a}\\
				\overline{T}_{i2}&\ge\frac{1}{||A||}\,\textrm{ln}\left(1+\frac{\nu}{||c||\>\rho||x(t_i)||+\mu}\right)\label{27b},
			\end{align}
		\end{subequations} 
		where $\rho$ and $\gamma$ are the same as those in Theorem \ref{thm2-bound}, and $\mu$ is given by 
		\begin{equation}
			\begin{aligned}
				\mu&\triangleq||B(c^TB)^{-1}||K+||B||d_\textrm{max}.
			\end{aligned}
	\end{equation}}
\end{theorem}
\begin{proof}
	At first, by referring to the proof of Theorem \ref{thm2-bound}, one can acquire that under the event-triggered SMC \eqref{7} with ETM \eqref{22}, it leads to
	\begin{equation}
		||e(t)||\le\frac{\rho||x(t_i)||+\gamma}{||A||}{\left(e^{||A||(t-t_i)}-1\right)}
	\end{equation}
	for $t\in[t_i,t_{i+1}), i\in \mathbb{N}$. By using the triggering rule \eqref{22}, one can get that 
	\begin{equation}\label{30}
		\textrm{sin}\frac{\theta}{2n}\,||x(t_i)||\le\frac{\rho||x(t_{i})||+\gamma}{||A||}\left({e}^{||A||\tilde{T}_{i1}}-1\right).
	\end{equation} 
	The lower bound for $\tilde{T}_{i1}$, which is given in \eqref{27a}, can be obtained from \eqref{30}.
	\indent Next, consider the set $\Gamma\!=\!\lbrace t\in(t_i,\infty):||e(t)||\!=\!0\rbrace$.
	For all $t\in[t_i,t_{i+1})\setminus\Gamma$, taking the time derivative of $e(t)$ 
	\begin{equation}\label{31}
		\begin{split}
			\frac{d}{dt}||e(t)||&\le\left|\left|\frac{d}{dt}e(t)\right|\right|=\left|\left|\frac{d}{dt}x(t)\right|\right|\\
			&={\bigg|\bigg|Ax(t_i)-Ae(t)-B(c^TB)^{-1}c^TAx(t_i)\Bigg.}\\ \notag&\quad-{\bigg.B(c^TB)^{-1}K\textrm{sign}s(t_i)+Bd(t)\bigg|\bigg|}\\
			&\le||A||\>||e(t)||+\left|\left|\left(A-B(c^TB)^{-1}c^TA\right)x(t_i)\right|\right|\\
			&\quad+||B(c^TB)^{-1}K||+||B||d_\textrm{max}\\
			&\leq||A||\>||e(t)||+\rho(||x(t_i)||)+\mu.
		\end{split}
	\end{equation}
	Using comparison lemma \cite{khalil2002nonlinear}, the solution to the differential inequality can be obtained. Specifically, the solution with the initial condition $||e(t_i)||=0$ is obtained as follows: 
	\begin{equation}
		||e(t)||\le\frac{\rho(||x(t_i)||)+\mu}{||A||}(e^{||A||(t-t_i)}-1)
	\end{equation}
	for all time $t\in[t_i,t_{i+1}), i\in \mathbb{N}$. Recalling the triggering rule \eqref{23}, it is clear that
	\begin{equation}
		\frac{\nu}{||c||\>||A||}\le\frac{\rho(||x(t_i)||)+\mu}{||A||}(e^{||A||\overline{T}_{i2}}-1),
	\end{equation}
	which yields \eqref{27b}. It is evident that when the system state is kept within the practical sliding mode cone, the right-hand sides of both \eqref{27a} and \eqref{27b} have a positive lower bound. Recalling $T_i=\textrm{max}(\tilde{T}_{i1},\overline{T}_{i2})$, this indicates that there is a globally unified lower bound on the inter-event time, namely, ETM \eqref{22}-\eqref{23} is a global ETM. Hence, the proof is completed.
\end{proof}
\begin{figure}[t]
	\centering
	\includegraphics[width=3.9in]{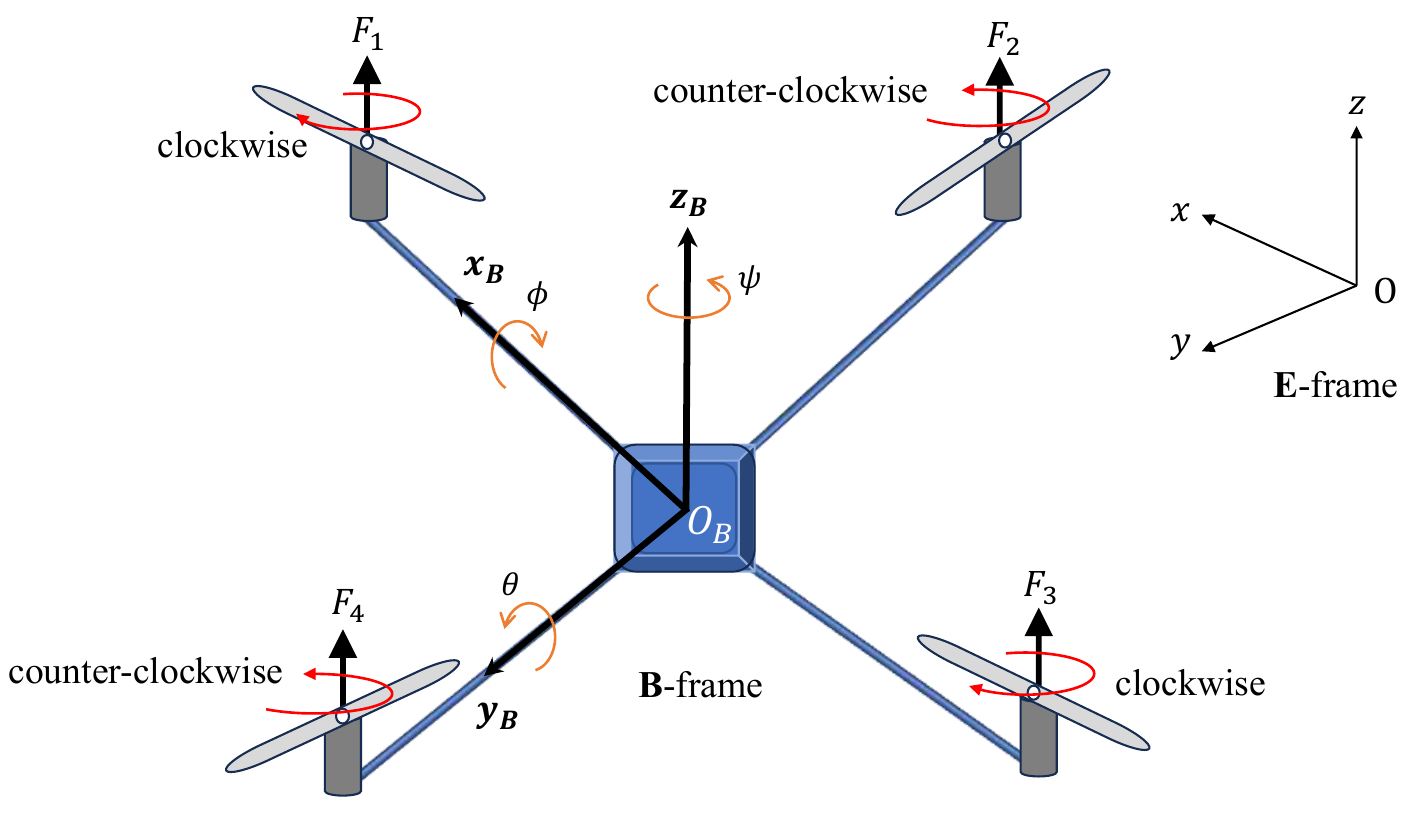} 
	%\vspace{5pt}
	\caption{The configuration of the quadrotor UAV.}\label{Fig. 5}
\end{figure}
\section{Application}
In this section, we consider an application of event-triggered SMC based on the practical sliding mode cone to the quadrotor UAV systems. As shown in Fig.~\ref{Fig. 5}, the four propellers are mounted at equal distances from the center of mass on the cross-shaped arms of the quadrotor (see \cite{7953502}). All propeller rotation axes are fixed and oriented downward. To generate upward lift, the propellers use fixed-pitch blades designed to direct airflow downward. The front and rear propellers (labeled as 1 and 3) rotate clockwise, while the left and right propellers (labeled 2 and 4) rotate counterclockwise. If all four propellers rotate at identical angular velocities, the resulting aerodynamic torques balance each other, producing a net yaw angular velocity of zero. The thrust forces produced by the propellers are indicated in Fig.~\ref{Fig. 5} as $F_i$.\\
\indent Specifically, the  linearized quadrotor UAV system can be modeled as follows \cite{7953502}:
\begin{equation}\label{33}
	\begin{aligned}
		\begin{bmatrix}
			\dot{x}_1(t)\\
			\dot{x}_2(t)\\
			\dot{x}_3(t)\\
			\dot{x}_4(t)\\
			\dot{x}_5(t)
		\end{bmatrix}&=\begin{bmatrix}
			0&1&0&0&0\\
			0&0&g&0&0\\
			0&0&0&1&0\\
			0&0&0&0&\frac{2K_ml}{I_{xx}}\\
			0&0&0&0&-\omega
		\end{bmatrix}\begin{bmatrix}
			x_1(t)\\
			x_2(t)\\
			x_3(t)\\
			x_4(t)\\
			x_5(t)
		\end{bmatrix}
		+\begin{bmatrix}
			0\\0\\0\\0\\\omega
		\end{bmatrix}(u(t)+d(t)),
	\end{aligned}
\end{equation}
where $x(t)=(x_1(t),x_2(t),x_3(t),x_4(t),x_5(t))^T$ represents the system state, indicating the position along the inertial $x$-axis, the linear velocity along the inertial $x$-axis, pitch angle, pitch angular velocity and actuator state for pitch dynamics, respectively. $g$ is the acceleration due to gravity and $I_{xx}$ is the moment of inertia around
$x_B$-axes when the quadrotor rotates around $x_B$-axes. Further, $\omega$ represents the actuator bandwidth, $K_m$ denotes a
positive gain, and $l$ is the distance between the actuator and the center of mass of the quadrotor. Here, $u(t)$ is the control input, and the matched external disturbance $d(t)$ satisfies the aforementioned condition.
Based on the above analysis, we construct the event-triggered SMC laws:
\begin{equation}\label{34}
	u_1(t)=-(\frac{1}{\omega})\big[({c}^TB)^{-1}\left({c}^TAx(t_i)+K(x(t_i))\textrm{sign}s({t}_i)\right)\big]
\end{equation}
\begin{equation}\label{35}
	u_2(t)=-(\frac{1}{\omega})\big[(c^TB)^{-1}(c^TAx(t_i)+K\textrm{sign}s(t_i))\big]
\end{equation}
the correlation coefficients for SMC laws \eqref{34} and \eqref{35} are similar to those for the event-triggered SMC laws \eqref{7} and \eqref{21}. So we can get the following result by using Theorem \ref{thm5-bound}.
%\begin{figure*}[!t]
%	\centering
%	\subfloat[System state evolution with event-triggered SMC.]
%	{\includegraphics[width=0.45\textwidth]{1.1state}}
%	\hfill
%	\subfloat[Evolution of inter-event time.]{\includegraphics[width=0.45\textwidth]{1.2time}}
%	\vspace{10pt}
%	\caption{Simulation results of Example 1.}\label{Fig. 7}
%\end{figure*}
\begin{theorem}\label{thm7-bound}
	{\rm Consider quadrotor UAV system \eqref{33} under event-triggered SMC laws \eqref{34} and \eqref{35} and the control algorithm is described by Fig. \ref{Fig. 4}. Then, the state of the addressed system is robustly ultimately bounded. Particularly, a globally unified lower bound on the inter-event time can be derived.}
\end{theorem}
\begin{proof}
		Owing to the proof of this theorem is analogous to that of Theorem \ref{thm5-bound}, it is omitted here.
\end{proof}

\begin{remark} 
{\rm The application of the event-triggered SMC based on practical sliding mode cone, as exemplified by the quadrotor UAV, not only underscores the practical relevance of our approach but also distinguishes it from existing methodologies. For instance, event-triggered SMC based on practical sliding mode band was used in \cite{https://doi.org/10.1002/rnc.7364} to stabilize the quadrotor UAV systems. Recently, \cite{POUZESH2022107337} investigated event-triggered fractional-order sliding mode control of disturbed quadrotor UAV systems. However, the global ETM was not well considered in these previous results. Namely, the positive lower bound is only guaranteed for inter-event time in semi-global sense in existing literature \cite{https://doi.org/10.1002/rnc.7364,POUZESH2022107337}. However, by using the concept of practical sliding mode cone, a global ETM is proposed in the current study, that is, the inter-event time can possess a unified positive lower bound from zero globally.}
\end{remark}
\section{Illustrative Examples}
In this section, some illustrating examples are presented to show the validity and advantages of the derived results. In particular, an application to quadrotor unmanned aerial vehicles is obtained.

\textit{Example 1:} 
\begin{figure}[t]
	\centering
	\subfigure[System state evolution with event-triggered SMC.]{\includegraphics[width=4.2in,height=2.6in]{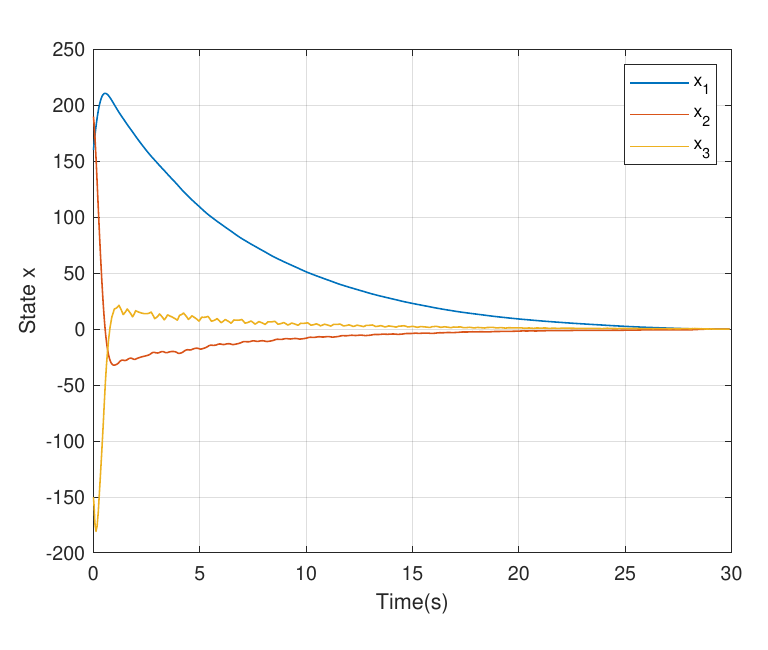}}
	\subfigure[Evolution of inter-event time.]{\includegraphics[width=4.2in,height=2.6in]{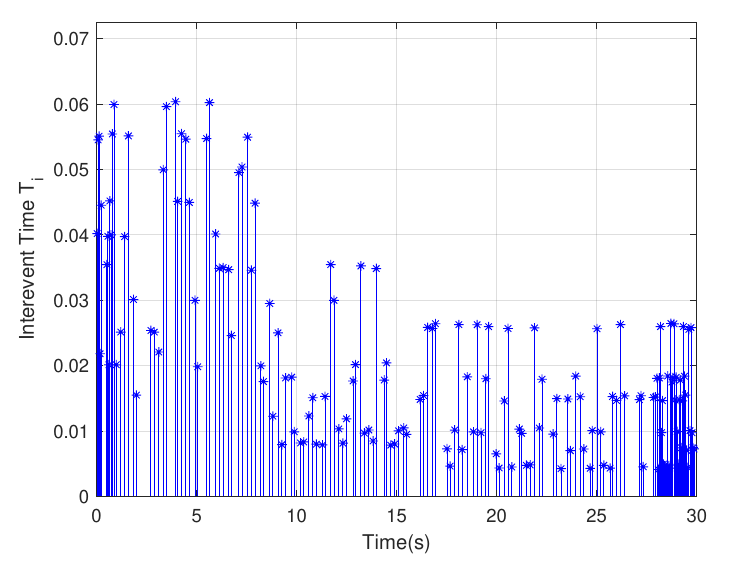}}
	\caption{Simulation results of Example 1.}
	\label{exa.1}
\end{figure}

\begin{figure}[t]
	\centering
	\includegraphics[width=3.9cm]{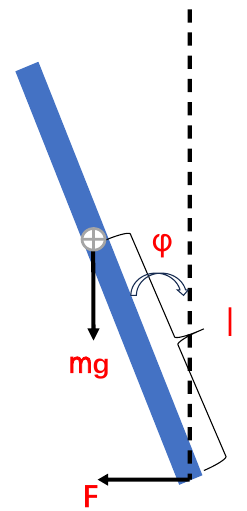} 
	%\vspace{5pt}
	\caption{The rotational single-arm pendulum.}\label{Fig. 6}
\end{figure}
\begin{figure}[t]
	\centering
	\includegraphics[width=4.2in,height=2.6in]{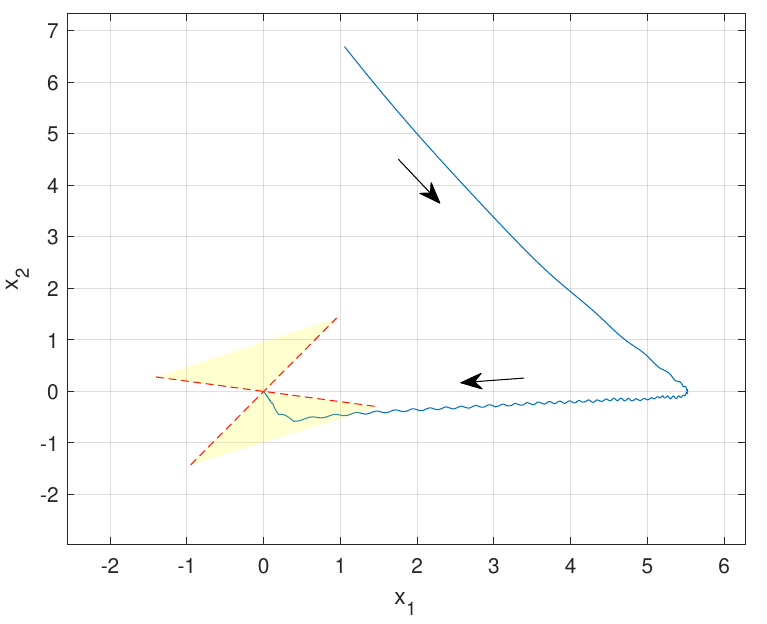} 
	%\vspace{5pt}
	\caption{Phase plane plot of the inverted pendulum system.}\label{exa.2}
\end{figure}
Consider the following LTI system:
\begin{equation*}
	\dot{x}=\left[\begin{array}{ccc}
		0 & 1 & 0\\ 
		0 & 2.1 & 4\\
		-1 & 2 & 3
	\end{array}\right]x+	\left[\begin{array}{c}
		0 \\ 0\\ 1
	\end{array}\right]\![u(t)+0.1\sin(10t)].
\end{equation*}
\indent The initial condition is given by $x_0=(160,190,-150)^T$. Design the sliding variable parameters $s=c^Tx,~c=(3.6,2,1)^T$; $\hat{s}=\hat{c}^Tx,~\hat{c}=(1.23,1.2,1)^T$; $\check{s}=\check{c}^Tx,~\check{c}=(7.4,0.9,1)^T$. Thus, the angel $\theta$ between $\hat{S}$ and $\check{S}$ is approximately equal to $\frac{173\pi}{225}\mathrm{rad}$. The following parameters are chosen by $n=24$, $\sigma=0.34$ and $\beta=0.48$. We can check that $\left|{c}^TB\right|d_\textrm{max}+\sigma||x(t_i)||+\beta=0.49+0.34||x(t_i)||$. Then, we can design event-triggered SMC \eqref{7} with $K(x(t_i))=1.79+0.49||x(t_i)||$ that satisfies condition \eqref{9}.\\
\indent Fig. \ref{exa.1} shows the performance of system under event-triggered SMC \eqref{7} with hybrid ETM \eqref{5}-\eqref{6}. As illustrated in Fig. \ref{exa.1}(a), the system states converge to the equilibrium point, which matches Theorem \ref{thm2-bound}. The inter-event time generated by ETM \eqref{5}-\eqref{6} is shown in Fig. \ref{exa.1}(b). Up to now, by using practical sliding mode band, some interesting works on event-triggered SMC and applications have been presented in \cite{behera2021survey,bandyopadhyay2018event,7048489,chen2025design,cucuzzella2020event,kumari2020event,liu2018event}. But these results can only guarantee that the state of system eventually converge to some neighborhood of the equilibrium point. Compared to these results based on practical sliding mode band, our result can ensures that the system state converges to the equilibrium point by using ideal sliding mode cone. In particular, \cite{liu2020event} also achieves the robust asymptotic stability by using event-triggered SMC with a well-designed dynamic ETM. However, the design of SMC laws requires additional state observer, which may increase the control cost and  computational complexity to some extent.

Actually, we not only achieve the global robust asymptotic stability by using the proposed event-triggered SMC, and the inter-event time has a unified lower bound when the state is relatively large and even infinite. Note that the lower bound for inter-event time can be derived semi-globally in \cite{behera2021survey,bandyopadhyay2018event,7048489}, that is, inter-event time turns to be zero when system states become infinite. Actually, one can observe from \ref{exa.1}(b) that the inter-event time in interval $[0,5]$, in which the system state is relatively large, possesses a positive lower bound by using the method in the current study. \\

\indent \textit{Example 2:} Consider system of rotational single-arm pendulum (see \cite{6360606})
\begin{equation*}
	\left[\begin{array}{c}
		\dot{x}_1(t) \\ 
		\dot{x}_2(t)
	\end{array}\right] 
	\!	=\! 
	\left[\begin{array}{cc}
		0 & 1\\ 
		\frac{mgl}{J} & 0
	\end{array}\right] 
	\left[\begin{array}{c}
		x_1(t) \\ 
		x_2(t)
	\end{array}\right] 
	\!	+\! 
	\left[\begin{array}{c}
		0\\ 
		\frac{l}{J}
	\end{array}\right][u(t) + d(t)],
\end{equation*}
which is shown in Fig. \ref{Fig. 6}. Specifically, $x_1(t)$ and $x_2(t)$ represent the angle of the pendulum, the angular speed, respectively. At the same time, $m$, $g$, $l$, and $J$ represent the mass of the pendulum and the acceleration of gravity, the distance between the center of gravity of the rod and the rotating shaft, and the damping coefficient, respectively. We consider the inverted pendulum system with initial condition $x_1(0)=\frac{\pi}{3}$, $x_2(0)=6.7$ and parameters $m=0.1$, $g=9.8$, $J=0.15$ and $l=0.3$. Here, the sliding parameters are chosen as $c=(2.1,1)^T$, $\hat{c}=(1.32,1)^T$, and $\check{c}=(4.1,1)^T$. The angel $\theta\approx\frac{87\pi}{100}\mathrm{rad}$ and $K(x(t_i))=0.95+1.8||x(t_i)||$. In particular, let external disturbance $d(t)=0.1\textrm{sin}(10t)$ and $n=23$. Using \eqref{9}, the triggering parameters are taken as $\sigma=0.03$ and $\beta=0.25$.\\
\indent The performance of the system with proposed control strategy is shown in Fig.~\ref{exa.2}. Unlike conventional continuous sliding mode control \cite{doi:10.1080/00207179208934240}, which requires persistent control signal updates and suffers from chattering phenomena near the sliding manifold, the proposed event-triggered scheme significantly reduces controller updates while maintaining comparable convergence performance. In comparison to the experimental outcomes presented in \cite{20161012}, our proposed event-triggered SMC achieves that the state trajectories converge to the equilibrium point, as shown in Fig. \ref{exa.2}. And one can see that the system state can reach and then remain in the ideal sliding mode cone region, which is shown in the yellow areas highlighted in Fig. \ref{exa.2}. Note that the previous results based on practical sliding mode band approach (e.g., \cite{7448882,behera2021survey,bandyopadhyay2018event,behera2025event,7048489,KUMARI20191}) can be only capable of driving the system state to a neighborhood  of the equilibrium point, without reaching the origin. Hence, we can determine that the inverted pendulum system is robustly asymptotically stable by applying the result of Theorem \ref{thm3-bound}. \\

\begin{figure}[t]
	\centering
	\subfigure[State evolution of the system \eqref{36} under event-triggered SMC.]{\includegraphics[width=4.2in,height=2.6in]{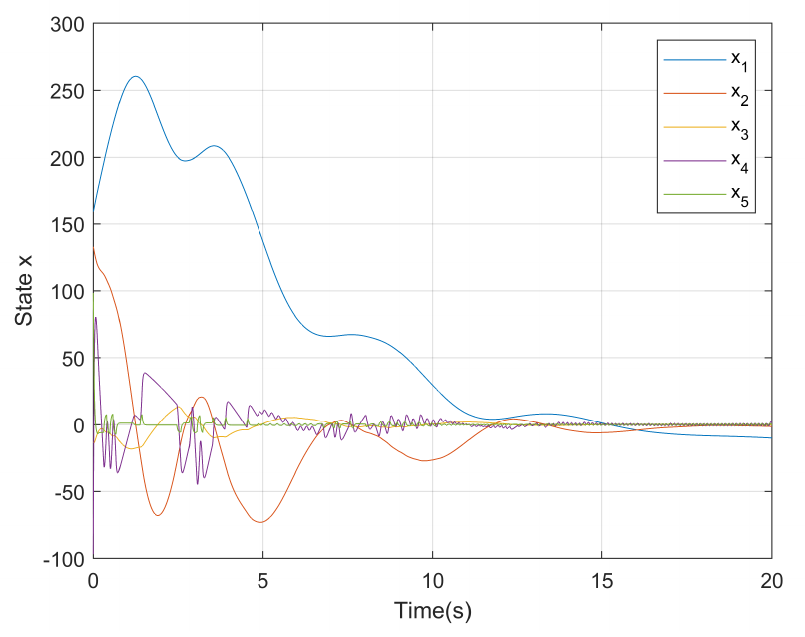}}
	\subfigure[Evolution of inter-event time.]{\includegraphics[width=4.2in,height=2.6in]{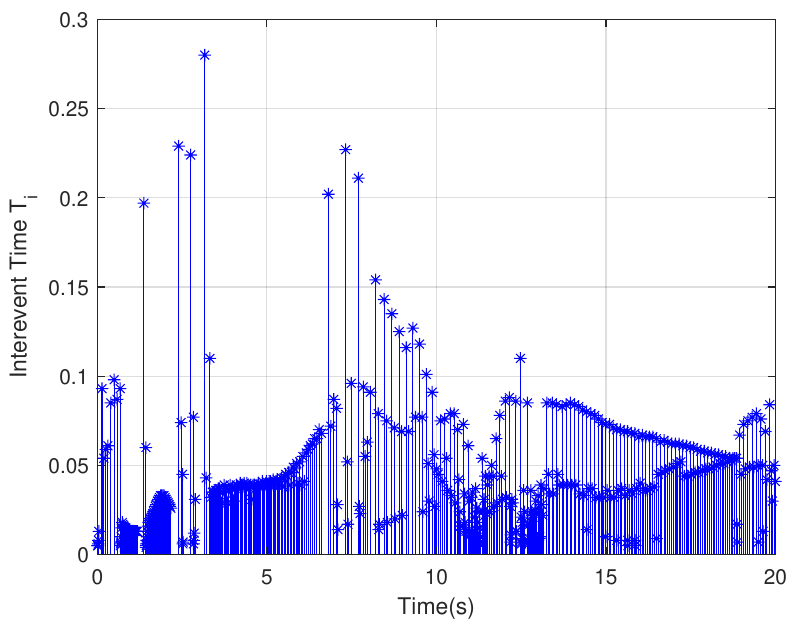}}
	\caption{Simulation results of Example 3.}
	\label{exa.3}
\end{figure}
\indent \textit{Example 3:} In this example, an application of derived results to quadrotor UAV system \ref{33} shall be obtained. To be specific, we consider quadrotor UAV system \ref{33} with parameters: $g=9.8\,\rm{m/s^2}$, $K_m=80\,\rm{N}$, $l=0.3\,\rm{m}$, $I_{xx}=0.6\,\rm{kg\cdot m^2}$, and $\omega=35\,\rm{rad/s}$. Then, the quadrotor UAV system can be given as follows:
\begin{equation}\label{36}
	\begin{aligned}
		\begin{bmatrix}
			\dot{x}_1(t)\\
			\dot{x}_2(t)\\
			\dot{x}_3(t)\\
			\dot{x}_4(t)\\
			\dot{x}_5(t)
		\end{bmatrix}&=\begin{bmatrix}
			0&1&0&0&0\\
			0&0&9.8&0&0\\
			0&0&0&1&0\\
			0&0&0&0&80\\
			0&0&0&0&-35
		\end{bmatrix}\begin{bmatrix}
			x_1(t)\\
			x_2(t)\\
			x_3(t)\\
			x_4(t)\\
			x_5(t)
		\end{bmatrix}
		+\begin{bmatrix}
			0\\0\\0\\0\\35
		\end{bmatrix}(u(t)+d(t)),
	\end{aligned}
\end{equation}
where the initial condition is given by $x_0=(159,133,-13,-105,102)^T$.
In addition, the disturbance is considered in the form of $d(t)=0.5\,\textrm{cos}(0.08\pi t)$. By selecting sliding variable parameters $c=(0.08,0.37,1.89,1.83,1)^T$, $\check{c}=(0.08,0.97,2.01,0.79,1)^T$ and $\hat{c}=(0.14,0.3,3,0.84,1)^T$, one can obtain that $\theta\approx\frac{161\pi}{180}\mathrm{rad}$. Set $\nu=643$, $\sigma=32$, $\beta=169$ and $n=26$ and the discrete time step for simulations is given by $dt=0.001\mathrm{s}$. We design the control gain $K(x(t_i))=915+35||x(t_i)||$, which satisfies condition \eqref{9}. By using Theorem \ref{thm7-bound}, we can infer that quadrotor UAV system \eqref{36} under the event-triggered SMC laws \eqref{34} and \eqref{35} with ETM \eqref{22}-\eqref{23} is ultimately bounded, which is shown in Fig. \ref{exa.3}(a). It is imperative to demonstrate that the inter-event time has a unified lower bound irrespective of the size of system state, as shown in Fig. \ref{exa.3}(b). The problem of whether a lower bound exists for inter-event times in systems under the infinite-state condition is not adequately addressed in \cite{7048489} and \cite{8460428}. Although the ETM based on practical sliding mode band proposed in \cite{https://doi.org/10.1002/rnc.7364} is designed to achieve distributed formation control for quadrotor UAV systems and reduce the consumption of network transmission resources under specific conditions under specific conditions, it overlooks the region of practical sliding mode cone and cannot ensure the lower bound for inter-event time globally. Moreover, in contrast to \cite{POUZESH2022107337}, our approach guarantees the existence of lower bound for the inter-event time, even in the case of an infinite system state. In particular, the inter-event time in $[0,5]$ (see Fig. \ref{exa.3}(b)) is relaxed in the presence of relatively large initial state, which matches the theoretical result derived in the current paper.

\section{Conclusion}
\label{conclusion}
This paper presents the ideal sliding mode cone concept for the first time, and based on this concept, a hybrid ETM is proposed to ensure global robust stability of LTI systems. By considering both the size of the error state and its direction of change, it is shown that the ETM enhances control accuracy and system robustness. Furthermore, it is proven that the system state will asymptotically converge to the equilibrium point, rather than just a neighbourhood of the equilibrium point, as indicated in previous studies. Additionally, inspired by the concept of practical sliding mode band, we extend our work and propose the concept of practical sliding mode cone to demonstrate that the unified lower bound for inter-event time can be ensured globally, that is, the proposed ETM is global. We also explore the practical application of our control strategy to quadrotor UAV systems, highlighting its versatility. The effectiveness of the proposed strategy is validated through both theoretical analysis and numerical simulations, which showcase significant improvements in system performance across various conditions.
%\section*{Acknowledgements}
%The authors would like to thank the Editor, the Associate Editor and the anonymous reviewers for their constructive comments and suggestions which improved the quality of this paper.
\bibliographystyle{siam}
\bibliography{ETSMC}

\end{document}